\documentclass[12pt]{amsart}
\usepackage{geometry} 
\usepackage{amsmath}
\usepackage{amssymb,latexsym} 
\usepackage{graphicx}
\usepackage{epstopdf}
\usepackage{array}
\usepackage{nicefrac}
\usepackage{xfrac}
\usepackage{youngtab} 
\usepackage{young}

\usepackage{color}
\newtheorem{theorem}{Theorem}
\newtheorem{definition}{Definition}

\newtheorem{proposition}{Proposition}
\newtheorem{lemma}[proposition]{Lemma}
\newtheorem{corollary}[proposition]{Corollary}
\newtheorem{remark}[proposition]{Remark}
\newtheorem{example}[proposition]{Example}
\title{Connected components of real $CB_{n}$ algebraic varieties}
\author{N. C. Combe}

 \address{Aix-Marseille University, I2M, CMI Technop\^ole Ch\^ateau-Gombert 39,rue F. Joliot Curie 13453 Marseille Cedex 13 France}
 
\email{noemie.combe@univ-amu.fr}

\thanks{The study was supported by a grant from the Labex Archimede and of the A*MIDEX project (ANR-11-IDEX-0001-02), funded by the ``Investissements d'Avenir'' French Government programme managed by the French National Research Agency (ANR).\\ 
I would like to thank Professor Norbert A'Campo for interesting discussions on this subject.\\ 
I would like to express all my gratitude to Professor Hanna Nencka for many valuable suggestions, discussions and comments.
}
\date{} 

\begin{document}
\maketitle{}

\begin{abstract}
Connected components of real algebraic varieties invariant under the $CB_{n}$-Coxeter group are investigated. 
In particular, we consider their maximal number and their geometric and topological properties.
This provides a decomposition for the space of $CB_{n}$-algebraic varieties.
We construct $CB_{n}$-polynomials using Young-posets and partitions of integers. Our results establish bounds on the number of connected components for a given set of coefficients.
It turns out that this number can achieve an upper bound of $2^{n}+1$ for specific coefficients. We introduce a new method to characterize the geometry of these real algebraic varieties, using J. Cerf and A. Douady theory for varieties with angular boundary and the theory of chambers and galleries. We provide several examples that bring out the essence of these results.  
 \end{abstract}

\section{Introduction}
\subsection{Motivation}
At the end of the 19th century geometers such as Clebsch , Goursat, Schl$\ddot{a}$fli, Klein investigated the geometry and the topology of real and complex algebraic surfaces.

Recently, there were many advances in the geometrical description of these surfaces.

In this paper we consider algebraic sets of type: \[Z(f)=\{Y\in \mathbb{A}_{\mathbb{R}}^n | f(Y)=0 , f\in \mathbb{R}[V]^{CB_{n}}\}\] with $CB_{n}$ the Coxeter group; $V$ the real vector space $\mathbb{R}^n.$ 
 
 \hspace{1pt}

It is well known by Hilbert's Nullenstellensatz~\cite{Ha}~\cite{Cor} that if $k$ is an algebraically closed field, an algebraic set is irreducible if and only if its ideal is a prime ideal. Somehow, if $k$ is no longer an algebraically closed field, there are no well established criterions to determine the number of irreducible components of the algebraic set.

Therefore this paper proposes to consider the number of connected components for a degree $2^{p}$ variety $Z(f)$, ($p$ is a positive integer) using a new method relying on chambers from Coxeter groups, J. Cerf and A. Douady theory on varieties with angular boundaries~\cite{{Do},{Ce61}}. 

\subsection{Main statement}
We will in particular show the following statement:

\textit{
Let $f\in\mathbb{R}[V]^{CB_{n}}$ be a degree four polynomial. Then, there exist subsets of $\mathbb{R}^4$ to which the set of coefficients of the polynomials belong and for which:
\begin{enumerate}\label{E:2}
\item the maximal number of connected components of the $n$-ary $CB_{n}$-quartic variety is $2^{n}+1$.
\item the  maximal number of compact connected components of the $n$-ary $CB_{n}$-quartic variety is $2^{n}$.
 \item  the maximal number of nested compact connected components of the $n$-ary $CB_{n}$-quartic variety is $2$.
\end{enumerate}
 }
 \hspace{1pt}
 
 In particular by nested compact connected components we mean concentric disjoint connected components isomorphic to a couple of disjoint concentric $n-1$-spheres.  

The next subsections are devoted to recalling some material for the convenience of the reader. 
We give a short overview on chambers and galleries and as well on $k$-sectors.

\subsection{Background}

Let us shortly recall some notions about invariant theory~\cite{Ol} and Coxeter groups theory~\cite{Bki}.
 The terminology used in this paper is defined in  \cite{Ol} and \cite{Bki}, \cite{Hu}. We will use these tools in our method to prove the results above. 
\subsubsection{Polynomials invariant under groups}

If $G$ is a finite group and $\rho: G \to GL_{n}(V)$ its representation, then $G$ acts on the vector space $V$ through linear transformation and this action extends to the ring of polynomials $\mathbb{R}[V]$ by the formula~\cite{Smi}
 \[ (gf)(v):=f(\rho( g^{-1}\ v)), \forall v\in V .\]
By $G$-invariant polynomials we mean a polynomial in the algebra $\mathbb{R}[V]^{G}$ defined by: 
\[ \mathbb{R}[V]^{G}:=\{f\in \mathbb{R}[V] \mid gf=f, \forall g\in G\} \] for a given group $G$. 
Note that the product and sum of two invariant polynomials under $G$ is invariant under $G$.
 A polynomial is invariant if and only if all its homogeneous components are invariant. 
 Moreover, $\mathbb{R}[V]^{G}$ is a graded algebra over $\mathbb{R}$.

\vspace{3pt}
\subsubsection{Chambers and galleries}

Let  $S=\{s_{1},\dots,s_{n}\}$ be a set of independent generators of a finite group $G$.  Let $\mathbf{m} =(m_{ij})_{1\leq i, j\leq n}$ be a symmetric $n\times n$ matrix with entries from $\mathbb{N}\cup \{\infty\}$ such that $m_{ii}=1$ for all $i\in \{1,\dots,n\}$  and $m_{ij}\geq 2$ whenever $i\ne j$.
The Coxeter group of type $\mathbf{m}$ is the group:
\[
G(\mathbf{m})=\langle \{s_{1},\dots,s_{n}\} \mid (s_{i}s_{j})^{m_{ij}}=1\ \forall i,j \in  \{1,\dots,n\},m_{ij}<\infty.\rangle.
\]
 The pair $(G,S)$ is called a Coxeter system with rank $ \#S$, note that  this pair is not unique.
 
 The graph associated to $(G,S)$ is the labeled graph with set of vertices $S$ and edges $\{s_{i},s_{j}\}$, labeled by $m_{i,j}$, when $m_{i,j}\geq 3$.

In the following we are interested by finite groups generated by reflections, more precisely by Coxeter finite groups.
 
Consider the finite reflection group $G$ in $V$, which is a finite subgroup of $GL_{n}(V)$, generated by reflections, where 
by reflection of $V$ we mean an automorphism of order 2 whose set of fixed points is an hyperplane. 
The reflection representation of $G$ is obtained as follows.

Define a bilinear form  on $V$ by $B_{M}(e_{s_{i}},e_{s_{j}})=-cos(\frac{\pi}{m_{i,j}})$ where $e_{s_{i}},e_{s_{j}}$ are the vectors of the canonical basis of $V$. The reflection on $V$ is given by: $\rho_{s}(x)=x-2B_{M}(e_{s},x)e_{s} $. The map $s\to\rho_{s}$ extends to an injective group morphism, $G\to GL_{n}(V)$, the reflection representation of $G$, see \cite{Bki} (Chap V, section 4.3 and 4.4). The reflection hyperplane is called a mirror.

Denote by $M_{G}$ the set of mirrors of $G$.
The connected components of the set $V-\cup_{H\in M_{G}} H$ are the chambers of $G$. The group $G$ acts simply transitively on the set of chambers \cite{Bki} (sectionV.3 theorem 1.2). 

By gallery of length $n$ we intend a sequence of adjacent chambers, i.e such that the intersection  of the closures of the chambers has codimension 1.
The closure of a chamber is a fundamental domain for the action of $G$ on $V$. 

\subsection{Sectors and faces}
Tools form J. Cerf's theory~\cite{Ce61} that we use are the following. 

\vspace{3pt}

We call a $k$-sector for the real vector space $V$ the subset of $V$ given by linear forms $f_{i}$ linearly independent,
such that $\{f_{1}\geq 0, \dots, f_{k}\geq 0\}$.

A codimension $j$-face is a set of $j$ equalities $\{f_{i_{1}}=0, \dots,f_{i_{j}}=0\}$ where $j\leq k$.
The index of a $k$-sector at the point 0  is the greatest codimension of the faces containing the point 0.

\begin{remark}
We interpret a chamber of the $CB_{n}$-Coxeter group as an $n$-sector with index $n-1$ at  0 and study the part of the algebraic variety which is in this fundamental domain.
The $CB_{n}$-algebraic variety is then obtained using reflections.
\end{remark}


\section{Polynomials invariant under $CB_{n}$ group}
In the following part, a method for the construction of polynomials $P_{d}\in \mathbb{R}[V]^{CB_{n}}$ is given.

The hyper-octahedral group $CB_{n}$ is the group of $n \times n$ monomial matrices with non-zero entries in $\{-1,1\}$, that is the signed permutation group.
 It is the symmetry group of the hypercube and of the dual cross-polytope.

The method to construct polynomials $P_{d}\in \mathbb{R}[V]^{CB_{n}}$ is as follows.
The key idea resides in the following theorem : Every $CB_{n}$-invariant polynomial in $k[x_{1}, \dots, x_{n}]$ can be written uniquely as a polynomial in the algebraically independent elementary polynomials $\sigma_{1}, \dots, \sigma_{n}\in k[x_{1},\dots, x_{n}]$: 

\[\begin{aligned}
\sigma_{1}&=\sum_{i=1}^{n}x_{i}^2\\
\sigma_{2}&=\sum_{i<j}x_{i}^2x_{j}^2\\
& \dots \\
\sigma_{r}&=\sum_{i_{1}< i_{2}<\dots i_{r}}x_{i_{1}}^2x_{i_{2}}^2\dots x_{i_{r}}^2 \\
& \dots \\
\sigma_{n}&=\prod_{i=1}^n x^2_{i}\\
\end{aligned}
\]
Notice that this choice is not unique.

To generate any homogeneous polynomial of degree $d=2q$ invariant under $CB_{n}$ with generators $\sigma_{i}$, it is useful to consider the partition $\lambda \dashv q$ of the integer $q$. The integers $\lambda=(\lambda_{1},\dots, \lambda_{k})$ correspond to the indexes of the elementary polynomials. Therefore, to obtain all the possible polynomials, we construct all the Young diagrams representing the partition of this integer, where we identify the partitions of type $q=i+j=j+i$, with $i$ and $j$ non zero.

The numbers of all the partitions $p(q)$ are given by the generating function  \[ \sum_{q=0}^{\infty}p(q)x^{q}=\prod_{k=1}^{\infty}\frac{1}{1-x^k}.\]

\begin{table}[h]
 \begin{tabular}{c|lllll|lllll}
  d=2q &\multicolumn{5}{c|}{Young  diagrams}&\multicolumn{5}{c}{Span of the vector space $\mathbb{R}[V]^{CB_{n}}$ }\\
&& && && & & &&\\
  2 &  $\tiny\yng(1)$ & & & && $\sigma_{1} $& & &&\\
  4 &$ \tiny\yng(2) $ & $\tiny\yng(1,1)$  & & && $\sigma_{2}$ & $ \sigma_{1}^2$ &  && \\
  6 & $\tiny\yng(3)$& $\tiny\yng(2,1)$& $\tiny\yng(1,1,1)$& &&  $\sigma_{3}$ & $\sigma_{2}\sigma_{1}$& $\sigma_{1}^3$ &&\\
  8 &$\tiny\yng(4)$& $\tiny\yng(3,1)$& $\tiny\yng(2,1,1)$& $\tiny\yng(2,2)$& $\tiny\yng(1,1,1,1)$& $\sigma_{4}$&$\sigma_{3}\sigma_{1}$&$ \sigma_{2}\sigma_{1}^2$&$\sigma_{2}^2$&$\sigma_{1}^4 $ \\    
 \dots &  \dots & \dots &  \dots& \dots& \dots & \dots & \dots &  \dots&  \dots& \dots\\
 \end{tabular}
 \caption{Homogeneous linearly independent $CB_{n}$-invariant polynomials associated to the $\sigma_{i} $}
 \end{table}

We are intersted in quartic polynomials $f\in \mathbb{R}[V]^{CB_{n}}$. 
These polynomials are a linear combination of the homogeneous polynomials $\sigma_{1}, \sigma_{1}^2, \sigma_{2}:{\tiny\yng(1)},\  {\tiny\yng(1,1)},\ {\tiny\yng(2)}, $ :
 \begin{equation}\label{E:1}
 \begin{aligned}
  f(x_{1},\dots x_{n})& = A\sigma_{2}  + B\sigma_{1}^2+ C\sigma_{1} +D \\ 
 &=A \sum_{i<j=1}^{n} x_{i}^2x_{j}^2 + B\left(\sum_{i=1}^{n} x_{i}^2\right)^2+ C\sum_{i=1}^{n} x_{i}^2 +D,\quad \{A,B,C,D\}\in \mathbb{R}.
 \end{aligned}
 \end{equation}
 
The zero-locus of such polynomials gives an $n$-dimensional manifold with $n^2$ symmetry $(n-1)$-planes.
Therefore it is sufficient to study the geometry of a given surface in a fundamental domain.
  
\section{Number of connected components and decomposition of $CB_{n}$ quartic varieties}

In this part, we give our results on the number of connected components for a quartic invariant by $CB_{n}$ group. The proofs are in the following section. 
\begin{definition}
We call geometric characteristic of the zero locus of a polynomial, 
a connected set of polynomials such that their set of connected components in the Euclidean space are equivalent up to isotopy. 
\end{definition}

\begin{theorem}[Upper bound on the number the connected components]

Suppose that $ f\in\mathbb{R}[V]^{CB_{n}}$ and that $f$ is of degree four. 
Then there exist sets of polynomials with given geometric characteristic such that:
\begin{enumerate}\label{E:2}
\item the maximal number of connected components is $2^{n}+1$.
\item the maximal number of compact connected components  is $2^{n}$
 \item the maximal number of nested connected components is $2$.
\end{enumerate}
\end{theorem}

\begin{corollary}
If the $n$-ary quartic has $2^{n}+1$ connected components, then for the induced topology from $\mathbb{R}^n$, there exists a unique compact connected component with center at the origin. The $2^{n}$ other connected components have infinite diameter.
\end{corollary}
By diameter of a subset we mean: 
 the least upper bound of the set of all distances between pairs of points in the subset.

\begin{corollary}
Let $b_{i}=dim H_{i}(X,\mathbb{A})$ be the $i$-th Betti number and $\mathbb{A}$ is a unital commutative ring.
 \begin{enumerate}
\item  If the quartic has $2^{n}+1$ connected components then $b_{0}=2^{n}+1, b_{n-1}=1, b_{i}=0$ for $i\neq \{0, n-1 \}$.
\item If the quartic has $2^{n}$ compact connected  components then $b_{0}=2^{n}, b_{n-1}=2^n, b_{i}=0$ for $i\neq \{0,n-1\} $
\item If the quartic has $2$  compact connected components then $b_{0}=2,b_{n-1}=2, b_{i}=0$ for $i\neq \{0,n-1\} $
\end{enumerate}
\end{corollary}

\vspace{3pt}

We remark the following.

\vspace{3pt}

Let $n=3$, and $X,Y$ be two 3-manifolds homeomorphic to a 3-ball in $\mathbb{R}^3$.
Consider two cases: 
\begin{enumerate}
\item $\mathcal{A}_{1}:=X \cup Y$ such that $X\cap Y=\emptyset$ and $X\not\subset Y$, $X \not\subset Y$.
\item  $\mathcal{A}_{2}:=X \cup Y$ such that $X \cap Y \neq\emptyset$ and  $X\subset Y$ , both concentric.
The algebraic surface $\partial{X}\amalg\partial{Y}$ in this case refers to a quartic defined by the coefficients: $A=1, B=-b, C=b, b\in\mathbb{R}^{+}_{*}$, $0<\beta=\frac{1}{|B|}<3$ and $0<k<1$.
   \end{enumerate}
 In both cases by~\cite{{Co1},{Co2},{Co3}}, $\partial{X}\amalg\partial{Y}$ have the same Betti number ; $b_{0}=2,b_{n-1}=2, b_{i}=0$ for $i\neq \{0,2\} $ .

  
\begin{remark}

 The same statement holds for arbitrary dimension $n$.
\end{remark}
 The proofs of both statement are straightforward.
\begin{proposition}
If $n=3$ then the $2^3+1$ is the maximal bound on the number of connected components of a $BC_{3}$-polynomial.
\end{proposition}
\begin{proof}
The proof follows from~\cite{{Co1},{Co2},{Co3}}.
\end{proof}
In the following proposition, we define a polynomial invariant under $CB_{n}$ of degree $d=2^{m+1}$ which has non-zero coefficients for the following  homogeneous polynomials invariant under $CB_{n}$: 
\[ u= \sum_{i<j=1}^{n} x_{i}^{2^m}x_{j}^{2^m} , v=(\sum_{i=1}^{n} x_{i}^{2^m})^2, w=\sum_{i=1}^{n} x_{i}^{2^m}. \] For the other homogeneous polynomials invariant under $CB_{n}$, the real coefficients are equal to zero. 
\begin{corollary}
If the invariant polynomial is of degree $d=2^{m+1}, m\in \mathbb{N}^{*}$ and of equation
 \begin{equation}\label{E:3}
 f(x_{1},\dots x_{n}) = A \sum_{i<j=1}^{n} x_{i}^{2^m}x_{j}^{2^m} + B\left(\sum_{i=1}^{n} x_{i}^{2^m}\right)^2+ C\sum_{i=1}^{n} x_{i}^{2^m} +D=0,  \end{equation} then the results in ~\eqref{E:2} hold for the same set of parameters $(A,B,C,D)\in \mathbb{R}^4$. 
 \end{corollary}
\section{Proofs}
We will list a few technical lemmas useful for the proofs.
 
Recall that the finite subgroup of $GL(V)$ associated to $CB_{n}$ admit:
\begin{enumerate}
\item  $n^2$ mirrors; 
\item $2^n n!$ chambers.
\end{enumerate}
This gives the number of fundamental domains and the number of reflections. Therefore, we can construct
 the algebraic variety and count the number of connected components. 

\begin{lemma}
 
Let $Q$ denote a quadric in the affine space  $\mathbb{R}^n$. 
Let the equation of $Q$ be given by the formula : $\bf{x}^{T}\Lambda\bf{x}=0$ with 
 
\[ \Lambda=\left(\renewcommand{\arraystretch}{1.5}\begin{tabular}{c|c}
$D$ &$ \frac{1}{2}\mathbf{C}^{T}$\\ \hline
$\frac{1}{2}\mathbf{C}$ & $\Lambda_{0}$  
\end{tabular}\right),
 \quad \mathbf{C}=(C,C, \dots, C)^{T},\ \Lambda_{0}= \begin{pmatrix}
B & W  & \dots &W\\ 
W & B & \dots & W\\
\dots & \dots & \dots & \dots\\\
W & W & \dots & B  \\  
\end{pmatrix}\quad W=\frac{A+2B}{2}.
 \] then the non-degenerated quadric has revolution axis $\{x_{1}= \dots = x_{n}\}$.
\end{lemma}

\begin{lemma}
Consider a chamber in $\mathbb{R}^n_{>0}$.
Suppose that $Q$ is a $(n-1)$-hyperboloid with strictly positive center. Then the  $(n-1)$-hyperboloid intersects the chamber in two disjoint connected components. 
\end{lemma}

\begin{lemma}
Let $V$ be the finite dimensional real vector space and $s$ an automorphism of $V$ of order two. 
Let $H$ be the set corresponding to $ker(s-1)$.

Let  $C$ be a connected subset of $V$.
\begin{itemize}
\item If $C \cap  ker(s-1)\neq \emptyset $, then after applying the automorphism $s$ there exists only one connected component $C'$ in $V$.
\item If $C \cap ker(s-1)= \emptyset$, then after applying the automorphism $s$ there exist two connected components in $V$.\end{itemize}
\end{lemma}
\begin{proof} 

\hspace{3pt}



\hspace{3pt}

We will prove theorem 1.1.
Let us consider $V=\mathbb{R}^n $.
The representation of $G=B_{n}$ generates $n^2$ reflection planes. 
The set of  $ker(s-1)$, where $s$ is a reflection of $G$ are defined by the set: 
$\mathcal{A}=\{ x_{1}=0,\dots , x_{n}=0,  x_{i}=\pm x_{j}\}, \space  i,j\in \{1,2,\dots n\}$.

Consider the substitution morphism defined in a chamber contained in $\mathbb{R}^n_{>0}$:  
\[\phi: \mathbb{R}^n_{>0}\to \mathbb{R}^n_{>0}
\]
\[  \bf{x} \mapsto \bf{ x^{ \frac{1}{2} } } \]

the polynomial $F=f\circ\phi$ is a quadric defined in the chamber. The morphism $\phi$ defined in $\mathbb{R}^n_{>0}$ is an isomorphism. Therefore, applying this isomorphism to the chamber, the topology of the surface in the chamber is not modified. 
The quadric polynomial is: 
   \[F(X_{1},\dots, X_{n}) = A \sum_{i<j=1}^{n} X_{i}X_{j}+ B\left(\sum_{i=1}^{n} X_{i}\right)^2+ C\sum_{i=1}^{n} X_{i} +D=0\]

Referring to ~\cite{Bu}, the classification of quadrics holds for the $n$-dimensional case. 

\vspace{3pt}

Assume that the quadric is non-degenerated and that it is a two -sheeted $n$-hyperboloid. In the notations of ~\cite{Bu} it corresponds to a canonical form defined as $(1,-1 \dots, -1)$ .

It is an easy calculation to show that the determinant : $\det(\Lambda_{0})=(B-W)^{n-1}(B+(n-1)W)=(\frac{-\beta}{2})^{n-1}(\frac{\beta(n-1)-2n}{2})$ is positive if $(n-1)$ is odd and negative if $(n-1)$ is even. 
Since the eigenvalues of  $\Lambda_{0}$ are all negative, we have the necessary condition: $0<\beta<\frac{2n}{n-1}$.

\vspace{3pt}

 Suppose that the center of the quadric is strictly positive , i.e $sgn(nC)=sgn(4((n-1)W -1))$. 
 Then, applying lemma 1 and 2 there exist two connected components in the chamber. 
 
 Choose coefficients $(A,B,C,D)$ such that the angle of the $n-1$-hyperboloid is acute, i.e there exists no intersection between the upper sheet of the $n-1$-hyperboloid and the $(n-1)$-face of the chamber, with support in the $n-1$-plane, defined by the equation $\{x_{i}=0\}$, for a given $i\in \{1,\dots,n\}$.
Therefore, in the chamber there exist two connected components.
 One is compact and the other of infinite diameter. 
 
 \vspace{3pt}
 
 Using the inverse morphism defined on the chamber in $\mathbb{R}^n_{>0}$: 
\[\phi': \mathbb{R}^n_{>0}\to \mathbb{R}^n_{>0}\]\label{E:5}
\[  \bf{x} \mapsto \bf{ x^2} \]

the $(n-1)$-surface is a quartic defined in the chamber in $\mathbb{R}^n_{>0}$ with two connected components. 
 \vspace{3pt}
Consider the lower connected component in the chamber. The lower connected component intersects all the $(n-1)$-faces of the chamber. The upper connected component intersects all the $(n-1)$-faces with support in the $(n-1)$-planes with equation $\{x_{i}=x_{j}\}$ and non intersection with the face of the chamber with support in $\{x_{i}=0\}$.  
Since the closure of a chamber is a fundamental domain,
 applying lemma 9 the lower component forms a compact connected component in $\mathbb{R}^n$. 
 The upper connected component forms a connected component in the gallery contained in $\mathbb{R}^n_{>0}$. 
 By part 2 of lemma 9 there exists a connected component in each region of the plane arrangement $\mathbb{R}^n-\cup_{H\in \mathcal{A'}}H$, $\mathcal{A'}:=\{x_{i}=0\}_{i=1,\dots, n}$ which is the image of the upper connected component of the quadric, after reflection. 
Therefore, the quartic has one compact connected component centered at the origin and connected components with infinite diameter in each region of $\mathbb{R}^n-\cup_{H\in \mathcal{A'}}H$. Thus, we have shown that we have $2^n+1$ connected components for this given set of parameters $(A,B,C,D)\in\mathbb{R}^4$.
\end{proof}

\vspace{3pt}

\begin{example}
For $n=3$, there exist $9$ mirrors, defined by the faces $A=\{ x_{1}=0, x_{2}=0, x_{3}=0,  x_{1}=\pm x_{2},  x_{1}=\pm x_{3},   x_{2}=\pm x_{3}\}$.
Obviously the three planes  $\{x_{1}=0, x_{2}=0, x_{3}=0\}$ partition the euclidean space into $8$ connected components and each of these 8 connected components 
are divided into 6 conical sub-regions.
Using the substitution map, one defines the following quadric:  \[F(X_{1},X_{2}, X_{3}) = A \sum_{i<j=1}^{3} X_{i}X_{j}+ B\left(\sum_{i=1}^{3} X_{i}\right)^2+ C\sum_{i=1}^{3} X_{i} +D=0\]
Setting the quadric in matrix notation, the determinants and eigenvalues of the matrices are:
 \begin{equation}
\det \Lambda = (B-W)^{2}\left((B+2W)D-\frac{3}{4}C^{2} \right)   ,\quad \det \Lambda_{0}=(B-W)^{2}(B+2W).
 \end{equation}
 The eigenvalues of $\Lambda_{0}$ are:
  \begin{equation}
  \lambda_{1}=\lambda_{2}=B-W=-\frac{A}{2}, \ \lambda_{3}=B+2W= A+3B.
  \end{equation}
From the normalized  eigenvectors $\{\mathbf{v}_{1},\mathbf{v}_{2},\mathbf{v}_{3}\}$, one can obtain the columns of the transition matrix $P$ : 
\[ P=\left( \begin{array}{ccc}
-\frac{1}{\sqrt2} & -\frac{1}{\sqrt2}  & \frac{1}{\sqrt3}\\ 
0 & \frac{1}{\sqrt2} & \frac{1}{\sqrt3}\\
\frac{1}{\sqrt2} & 0 & \frac{1}{\sqrt3}  \\  
\end{array}\right).\]
An hyperboloid corresponds to the case where $\frac{3}{3-\beta}< k<0$. By lemma 6, the revolution axis is the line $\{x_{1}=x_{2}=x_{3}\}$. An easy calculation shows that the hyperboloid is acute , i.e the upper sheet does not intersect the face of the chamber supported  by a given plane: $\{x_{i}=0\}$. Using the method of the proof one obtains $2^3+1$ connected components.
\end{example}

\begin{proof}[Theorem1.2]

\vspace{3pt}

The method to prove the assertion of theorem 1.2 is similar to the precedent one. 
By the substitution isomorphism defined in a chamber in $\mathbb{R}^n_{>0}$:
\[\phi: \mathbb{R}^n_{>0}\to \mathbb{R}^n_{>0}
\]
\[  \bf{x} \mapsto \bf{ x^{ \frac{1}{2} } } \]
we define a quadric $F=f\circ\phi$ defined on the chamber. 
By lemma 1, the center of the quadric is on the 1-dimensional set $\{x_{1}=x_{2}=\dots=x_{n}\}$.
Referring to the classification~\cite{Bu}, and assuming that the quadric is non degenerated with canonical form $(1,1,\dots,-1)$: this corresponds to an $(n-1)$-ellipsoid. Suppose that the center is strictly positive and that the compact quadric does not intersect the $(n-1)$-faces supported by $n$-planes of equation $ \{ x_{i}=0\}$ for given $i\in \{1,n\}$. On the other side, the ellipsoid intersects the $(n-1)$-faces supported by $x_{i}=x_{j}$. The inverse substitution map defines the quartic in the chamber . Therefore by lemma 9, the component in $\mathbb{R}^n_{+,+}$ is compact and does not intersect the $n$-planes in $\mathcal{A'}$. Since this hyperplane arrangement divides $V$ into $2^n$ regions, by part 2 of lemma 9 there exist $2^n$ components that are compact. 
\end{proof}
\begin{proof}[Theorem 1.3]

The method of the proof is the same as previously.
We consider the quadric defined in a chamber contained in $\mathbb{R}^n_{>0}$ and assume that the quadric is an $n$-ellipsoid with strictly positive center. 
Moreover, we assume that the $n$-ellipsoid intersects all of the $(n-1)$-planes of the chamber. Therefore, since the ellipsoid is of degree 2 and a plane of degree 1, both intersect in two $(n-2)$-dimensional manifolds, which are disjoint, in our case. After applying the substitution map and by lemma 9.1, the resulting surface is a disjoint union of two compact connected components which are concentric.  \end{proof}

\begin{proof}[Corollary 2, 3 \& Proposition 4,5]
\vspace{3pt}

\begin{enumerate}
\
\item $Corollary\ 2$ follows from the proof of theorem 1.1 
\item $Corollary\  3$ follows from the proofs of theorem 1.
\item $Proposition\ 4$ follows from~\cite{Co1}
\item $Proposition\ 5$ : the quadric defined in (\ref{E:3}) is an invariant polynomial under 
$CB_{n}$ action. The substitution isomorphism $\phi$ can be defined in a chamber as $\bf{x}\mapsto \bf{x}^{\frac{1}{2^m}}$ in $\mathbb{R}_{>0}^n$. Since this isomorphism does not modify the topology of the hyper-surface defined in the chamber, therefore we can use the same idea as in the theorem 1 and discuss the type of the quadric in the chamber following~\cite{Bu}. So the results for quadrics in theorem 1 hold in the case of invariant polynomials defined in (\ref{E:3}) with degree $2^{k}$ .
\end{enumerate}
\end{proof}
\begin{proof}[lemma 6]
Let us prove lemma 6. The quadric is defined in  matrix notation by: 
\[ \Lambda=\left(\renewcommand{\arraystretch}{1.5}
\begin{tabular}{c|c}
$D$ &$ \frac{1}{2}\mathbf{C}^{T}$\\ \hline
$\frac{1}{2}\mathbf{C}$ & $\Lambda_{0}$  
\end{tabular}\right),
 \quad \mathbf{C}=(C,C,\dots, C)^{T},\ \Lambda_{0}= \begin{pmatrix}
B & W  & \dots &W\\ 
W & B & \dots & W\\
\dots & \dots & \dots & \dots\\\
W & W & \dots & B  \\  
\end{pmatrix},\quad W=\frac{A+2B}{2}.
 \]
 
Any easy computation shows that the eigenvalues of $\Lambda_{0}$ are:
  \begin{equation}
  \lambda_{1}=\lambda_{2}=\dots =\lambda_{n-1}= B-W=-\frac{A}{2}, \ \lambda_{n}=B+(n-1)W.
  \end{equation}
 From the normalized  eigenvectors $\{\mathbf{v}_{1},\mathbf{v}_{2},\dots, \mathbf{v}_{n}\}$, one can obtain the columns of  the  transition matrix $P$ : 
\[ P=\left( \begin{array}{cccc}
-\frac{1}{\sqrt2} & -\frac{1}{\sqrt2} & & \frac{1}{\sqrt n}\\ 
0 & \frac{1}{\sqrt2} & &\frac{1}{\sqrt n}\\

\dots & \dots & \dots & \dots \\
\frac{1}{\sqrt2} & 0 & & \frac{1}{\sqrt n}  \\  
\end{array}\right).\]
The eigenplane of the $(n-1)$ eigenvalues  $\lambda_{1}= \dots = \lambda_{n-1}=-\frac{A}{2}$ is orthogonal to the axis  which corresponds  to simple eigenvalue $\lambda_{n}=B+(n-1)W$, so the quadric has the line generated by the vector $(1,1,\dots,1)$ as revolution axis. 
The center of the quadric is given by: 
\[\left(\frac{nC}{4((n-1)W-1)},\dots, \frac{nC}{4((n-1)W-1)}\right).\]  

Therefore the revolution axis is $\{ x_{1}=x_{2}=\dots =x_{n} \}$.
\end{proof}

\begin{proof}[lemma 7]
By lemma 6 the quadric has revolution axis and center on $\{ x_{1}=x_{2}=\dots =x_{n} \}$.
So, if the coordinates of the center are strictly positive then the $n$-hyperboloid intersects the chamber in two disjoint connected components. 
\end{proof}
\begin{proof}[lemma 8]
The proof is straightforward.
\end{proof}

\end{document}